\documentclass[10pt]{article}
\usepackage{amsmath}
\usepackage{amsfonts}
\usepackage{amssymb}
\usepackage{graphicx}
\usepackage{mathrsfs}
\usepackage{xcolor}
\usepackage{bbm}
\usepackage{verbatim}
\usepackage{dsfont}
\usepackage{mathrsfs}
\usepackage[body={15.5cm,21cm}, top=3cm]{geometry}
 \usepackage{paralist}
\usepackage{indentfirst}
\allowdisplaybreaks[4]
\usepackage{hyperref}
\hypersetup{
colorlinks=true,
linkcolor=blue,
anchorcolor=blue,
citecolor =blue }
\providecommand{\U}[1]{\protect \rule{.1in}{.1in}}
\numberwithin{equation}{section}
\newtheorem{Theorem}{Theorem}[section]

\newtheorem{corollary}[Theorem]{Corollary}

\newtheorem{lemma}[Theorem]{Lemma}

\newtheorem{remark}[Theorem]{Remark}

\newenvironment{proof}[1][Proof]{\noindent \textbf{#1.} }{\  \rule{0.5em}{0.5em}}
\allowdisplaybreaks
\begin{document}
\title{A robust stochastic control problem with applications to monotone mean-variance problems}
\author{Yuyang Chen \thanks{School of Mathematical Sciences, Shanghai Jiao Tong University, China (cyy0032@sjtu.edu.cn)}
\and
Tianjiao Hua \thanks{School of Mathematical Sciences, Shanghai Jiao Tong University, China (htj960127@sjtu.edu.cn)}
\and
Peng Luo \thanks{School of Mathematical Sciences, Shanghai Jiao Tong University, China (peng.luo@sjtu.edu.cn)}}
\maketitle
\begin{abstract}
This paper studies a robust stochastic control problem with a monotone mean-variance cost functional and random coefficients. The main technique is to find the saddle point through two backward stochastic differential equations (BSDEs) with unbounded coefficients. We further show that the robust stochastic control problem shares the same optimal control and optimal value with the stochastic control problem with a mean-variance cost functional. The results obtained are then applied to monotone mean-variance and mean-variance portfolio selection problems and monotone mean-variance and mean-variance investment-reinsurance problems.
\end{abstract}
\textbf{MSC-classification}: 91B16, 93E20, 60H30, 91G1.
\section{Introduction}

The mean-variance (MV) preferences, initially introduced by Markowitz \cite{markowits1952portfolio}, stand as one of the most widely used objectives for portfolio selection problems. Many studies have utilized this criterion to investigate the portfolio selection problems, see e.g., \cite{zhou2000continuous,zhou2003markowitz,hu2005constrained} in continuous-time settings and \cite{leippold2004geometric,liang2008optioned,zhu2004risk} in discrete-time settings. However, a major drawback of MV preferences is their lack of monotonicity. That means a strictly wealthier investor may possess lower utility compared to others,  contradicting the fundamental assumption of economic rationality. To address this significant limitation, Maccheroni et al. \cite{maccheroni2009portfolio} introduced  monotone mean-variance (MMV) preferences and tackled the MMV model in a single-period setting. They showed that the MMV preference constitutes the smallest monotone preference functional coinciding with mean–variance on their domain of monotonicity. This provided an approach to show the optimal strategies to MV problems and MMV problems coincide by proving the optimal strategies to both MV and MMV preferences lead to a terminal wealth within the domain of monotonicity of the classical mean–variance functional. Within this framework, Strub and Li \cite{strub2020note} illustrated the optimal strategies and values of MV and MMV are consistent when asset prices are continuous. Further, in a stochastic factor model, Li, Liang and Pang \cite{li2022comparison} demonstrated the equivalence of MV and MMV under certain conditions, including situations involving jumps. Recently, Du and Strub \cite{du2023monotone} took the trading constraints into account and investigated the equivalence of MV and MMV portfolio selection problems when asset prices are continuous under general trading constraints.

On the other hand, to the best of our knowledge, Trybu{\l}a and Zawisza \cite{trybula2019continuous} first considered a continuous-time stochastic factor model, obtained the optimal strategy and value function for both the MV  and MMV optimization problems in an explicit
form, and demonstrated their equivalence through direct comparison. Following the same direction as \cite{trybula2019continuous}, Shen and Zou \cite{shen2022cone} considered a deterministic coefficients model with cone trading constraints and solved the MMV and the MV problems successively by means of the HJBI equation approach and verified that the optimal strategies to both problems coincide. Then, Hu, Shi and Xu \cite{hu2023constrained} generalized \cite{shen2022cone} to a diffusion model with random appreciation rate, volatility and deterministic interest rate. Recently, Shi and Xu \cite{xu2024} reach the same conclusion for the MMV and MV investment-reinsurance problems. In the conclusion of \cite{hu2023constrained}, they pointed out that how to solve the MMV problem when the interest rate is random even without any constraints constitutes an interesting question for future research.

In this paper, we consider a robust control problem with monotone mean-variance cost functional and the following state dynamics
$$
dX_{t} = (A_{t}X_{t}+u^{\prime}_{t}B_{t})dt+(X_{t}C_{t}^{\prime}+u^{\prime}_{t}D_{t})dW_{t},~~~X_0=x\in\mathbb{R},
$$ 
where all coefficients are random and the coefficient $A$ is even unbounded. The main technique is to find the saddle point through two backward stochastic differential equations (BSDEs) with unbounded coefficients. It is worth noting that the coefficients in our setting are all random and one of them is even unbounded, which makes our BSDEs significantly different from those in \cite{hu2023constrained} and cannot be solved by existing results on BSDEs. Relying on BMO martingale techniques, we establish some a priori estimates and successfully obtain the existence and uniqueness results of these two BSDEs. We further study the corresponding stochastic control problem with mean-variance cost functional and achieve that both the robust stochastic control problem and the stochastic control problem with mean-variance cost functional share the same optimal control and optimal value. Finally we apply our results to MV an MMV portfolio selection problems. We obtain optimal strategies for both MV and MMV problems with random coefficients in a complete market and show by direct comparison that they coincide. To the best of our knowledge, this is the first result of obtaining the explicit solution of the MMV problems with random (and unbounded) interest rate, which partially answer the question proposed in \cite{hu2023constrained} for random (and unbounded) interest rate and complete market.  Upon this basis, we further investigate MMV and MV investment-reinsurance problems and obtain the same conclusion. 

The reminder of this paper is organized as follows. In section \ref{section 2}, we formulate the robust control problems with random coefficients. In section \ref{section 3}, we establish the existence and uniqueness of solutions to two BSDEs with unbounded coefficients and obtain the optimal pair of the robust control problem. In section \ref{section 4}, we solve the control problem with mean-variance cost functional and make a comparison with the robust control problem. Applications to MMV and MV portfolio selection problems and MMV and MV investment-reinsurance problems are presented in section \ref{section 5}.
\section{Formulation of the problem}\label{section 2}

Let $(\Omega,\mathcal{F},\mathbb{P})$ be a complete probability space on which a standard $n$-dimensional Brownian motion $W=\{W_{t}\}_{t\geqslant 0}$ is defined and denote by $\mathbb{F}=\left\{\mathcal{F}_t\right\}_{t \geqslant 0}$ the usual augmentation of the natural filtration generated by $W$, by $\mathbb{E}[\cdot]$ the expectation under $\mathbb{P}$ and by $\mathbb{E}_t[\cdot]$ the conditional expectation under $\mathbb{P}$ given that $\mathcal{F}_t$.

Let $\mathbb{R}^{n}$ be the set of $n$-dimensional column vectors, let $\mathbb{R}_{+}^{n}$ be the set of $n$-dimensional column vectors whose components are nonnegative and let $\mathbb{R}^{m\times n}$ be the set of $m \times n$ real matrices. We denote the transpose of matrix $M$ by $M^{\prime}$, and the norm by $|M|=\sqrt{\operatorname{trace}\left(M M^{\prime}\right)}$. Let $\mathbb{S}^{n}$ be the set of symmetric
$n \times n$ real matrices. We write $M>$ (resp., $\geqslant) 0$ for any positive definite (resp., positive semidefinite) matrix $M \in \mathbb{S}^{n}$. We write $A>(\geqslant) B$ if $A, B \in \mathbb{S}^{n}$ and $A-B>(\geqslant) 0$. Unless otherwise stated, all equalities and inequalities between random variables and processes will be understood in the $\mathbb{P}$-a.s. and $\mathbb{P} \otimes d t$-a.e. sense, respectively.

We introduce the following spaces of random processes: for Euclidean space $\mathbb{H}=\mathbb{R}^{n},\mathbb{R}^{n\times n}$ and all $\mathbb{F}$-stopping time $\tau\leqslant T$ and under $\mathbb{P}$:
\begin{align*}
&L_{\mathbb{F}}^0(0, T ; \mathbb{H})=\left\{\varphi:[0, T] \times \Omega \rightarrow \mathbb{R} \mid \varphi \text { is an }\left\{\mathcal{F}_t\right\}_{t \geqslant 0} \text {-adapted process }\right\};\\
& L_{\mathbb{F},\mathbb{P}}^2(0,T;\mathbb{H})=\bigg\{\varphi \in L^{0}_{\mathbb{F}}(0,T;\mathbb{H})\mid\mathbb{E}\left[\int_0^T|\varphi(t)|^2dt\right]<\infty\bigg\};\\
&L_{\mathbb{F},\mathbb{P}}^2(\Omega;C([0,T];\mathbb{H}))=\bigg\{\varphi \in L^{0}_{\mathbb{F}}(0,T;\mathbb{H})\mid\mathbb{E}\left[\sup_{t\in[0, T]}|\varphi(t)|^2\right]<\infty\bigg\};\\
&L_{\mathbb{F},\mathbb{P}}^2\left(\Omega;L^1(0,T;\mathbb{H})\right)=\bigg\{\varphi \in L^{0}_{\mathbb{F}}(0,T;\mathbb{H})\mid\mathbb{E}\left[\left(\int_0^T|\varphi(t)|dt\right)^2\right]<\infty\bigg\};\\
&L_{\mathbb{F},\mathbb{P}}^{\infty}(0,T;\mathbb{H})=\bigg\{\varphi \in L^{0}_{\mathbb{F}}(0,T;\mathbb{H})\mid\varphi(\cdot)\text{ is }\text{essentially bounded}\bigg\};\\
&L_{\mathbb{F},\mathbb{P}}^{2,\mathrm{BMO}}(0,T;\mathbb{H})=\bigg\{\varphi \in L^{0}_{\mathbb{F}}(0,T;\mathbb{H})\mid\|\varphi\|_{\mathrm{BMO}_{2}}:=
\sup_{0\leqslant\tau\leqslant T}\left\|\mathbb{E}\left[\int_{\tau}^{T}|\varphi(s)|^{2}ds\mid\mathcal{F}_{\tau}\right]\right\|^{\frac{1}{2}}_{\infty}<\infty\bigg\}.
\end{align*}
We now introduce the following scalar-valued linear stochastic differential equation (SDE):
\begin{equation}\label{SDE}
    dX_{t} = (A_{t}X_{t}+u^{\prime}_{t}B_{t})dt+(X_{t}C_{t}^{\prime}+u^{\prime}_{t}D_{t})dW_{t},~~~~X_0=x\in\mathbb{R}.
\end{equation}
The class of admissible controls is defined as the set
\begin{align*}
\Pi:=L_{\mathbb{F},\mathbb{P}}^{2}\left(0,T;\mathbb{R}^{n}\right).
\end{align*}
And we will denote by $X^{u}$ the state process \eqref{SDE} whenever it is necessary to indicate its dependence on $u\in\Pi$. \\
For the coefficients, we assume the following conditions.\\
\textbf{Assumption (A)} $A\in L_{\mathbb{F},\mathbb{P}}^{2,\mathrm{BMO}}(0,T;\mathbb{R})$, $B,C\in L_{\mathbb{F},\mathbb{P}}^{\infty}(0,T;\mathbb{R}^{n})$, $D\in L_{\mathbb{F},\mathbb{P}}^{\infty}(0,T;\mathbb{R}^{n\times n})$. Moreover, there exists a positive constant $\delta$, such that
\begin{equation*}
    D_{t}D_{t}^{\prime} \geqslant \delta I_{n\times n}.
\end{equation*}
Here $I_{n\times n}$ is the $(n\times n)$-identity matrix.\\
For any process $\eta \in L_{\mathbb{F}}^{0}\left(0, T ; \mathbb{R}^{n}\right)$ such that
\begin{equation}\label{Lambda1}
\Lambda_{t}^{\eta}:=\mathcal{E}\left(\int_{0}^{t} \eta^{\prime} d W\right)
\end{equation}
is a martingale, define $\mathbb{P}^{\eta}$ by
\begin{align*}
\left.\frac{d\mathbb{P}^{\eta}}{d\mathbb{P}}\right|_{\mathcal{F}_{t}}=\Lambda_{t}^{\eta}.
\end{align*}
By Girsanov's theorem,
\begin{equation}\label{Girsanov}
W_{t}^{\eta}:=W_{t}-\int_{0}^{t}\eta_{s}ds
\end{equation}
is a Brownian motion under $\mathbb{P}^{\eta}$. Denote
\begin{align*}
\mathcal{A}:=\left\{\eta \in L_{\mathbb{F}}^{0}\left(0,T;\mathbb{R}^{n}\right) \mid \mathbb{E}\left[\Lambda_{t}^{\eta}\right]=1 \text { and } \mathbb{E}\left[\left(\Lambda_{t}^{\eta}\right)^{2}\right]<\infty \text { for all }t\in[0,T]\right\}.
\end{align*}
Let us now state our robust control problem as follows:
\begin{equation}\label{MMV1}
\sup_{u\in\Pi}\inf_{\eta\in\mathcal{A}} \mathbb{E}^{\mathbb{P}^{\eta}}\left[X_{T}^{u}+\frac{1}{2\theta}\left(\Lambda_{T}^{\eta}-1\right)\right],
\end{equation}
where $\theta$ is a given positive constant. The robust control problem is well-defined, according to the following lemma which give the solvability of state equation \eqref{SDE}.
\begin{lemma}\label{state space theorem}
    Under assumption (A), for any $u\in L_{\mathbb{F},\mathbb{P}}^2(0,T;\mathbb{R}^{n})$, the equation \eqref{SDE} admits a unique solution $X\in L_{\mathbb{F},\mathbb{P}}^2(\Omega;C([0,T];\mathbb{R}))$.
    \end{lemma}
    \begin{proof}
    First, given $\bar{X}\in L_{\mathbb{F},\mathbb{P}}^2(\Omega;C([0,T];\mathbb{R}))$ and $\bar{X}_{0}=x$, we consider
    \begin{equation}\label{SDE2}
    \left\{\begin{array}{l}
    dX_{t}=\left(A_{t}\bar{X}_{t}+u_{t}^{\prime}B_{t}\right)dt+\left(\bar{X}_{t}C_{t}^{\prime}+u_{t}^{\prime}D_{t}\right)dW_{t},\\
    X_{0}=x.
    \end{array}\right.
    \end{equation}
    From \cite[Lemma 1.6]{DelbaenandTang2010}, we know $A_{t}\bar{X}_{t}\in L_{\mathbb{F},\mathbb{P}}^2\left(\Omega;L^1(0,T;\mathbb{R})\right)$. Therefore the equation \eqref{SDE2} admits a unique solution $X\in L_{\mathbb{F},\mathbb{P}}^2(\Omega;C([0,T];\mathbb{R}))$ due to \cite[Proposition 2.1]{SunandYong2014}. Now we are going to prove $\Theta\big(\bar{X}\big):=X$ is a contraction map. For $\bar{X}^{1},\bar{X}^{2}\in L_{\mathbb{F},\mathbb{P}}^2(\Omega;C([0,T];\mathbb{R}))$ with $\bar{X}^1_0=\bar{X}^2_0=x$, we denote
    \begin{align*}
    X^{1}=\Theta\big(\bar{X}^{1}\big),\quad X^{2}=\Theta\big(\bar{X}^{2}\big)
    \end{align*}
    and
    \begin{align*}
    \Delta X=X^{1}-X^{2},\quad\Delta \bar{X}=\bar{X}^{1}-\bar{X}^{2},
    \end{align*}
    and obtain the following SDE:
    \begin{equation}\label{SDE3}
    \left\{\begin{array}{l}
    d\Delta X_{t}=A_{t}\Delta\bar{X}_{t}dt+\Delta\bar{X}_{t}C_{t}^{\prime}dW_{t},\\
    \Delta X_{0}=0.
    \end{array}\right.
    \end{equation}
    For any $0<\varepsilon<T$, from \cite[Lemma 1.6]{DelbaenandTang2010} we have
    \begin{equation*}
    \begin{aligned}
    \mathbb{E}\bigg[\sup_{0\leqslant t\leqslant \varepsilon}\big|\int_{0}^{t}A_{s}\Delta\bar{X}_{s}ds\big|^{2}\bigg]
    &\leqslant\mathbb{E}\bigg[\big|\int_{0}^{\varepsilon}|A_{t}|~\sup_{0\leqslant s\leqslant \varepsilon}|\Delta\bar{X}_{s}|dt\big|^{2}\bigg]\\
    &\leq8~\|A\|^{2}_{\mathrm{BMO}_{2}}~\mathbb{E}\bigg[\int_{0}^{\varepsilon}\sup_{0\leqslant s\leqslant \varepsilon}|\Delta\bar{X}_{s}|^{2}dt\bigg]\\
    &\leq8\varepsilon~\|A\|^{2}_{\mathrm{BMO}_{2}}~\mathbb{E}\bigg[\sup_{0\leqslant t\leqslant \varepsilon}\big|\Delta\bar{X}_{t}\big|^{2}\bigg].
    \end{aligned}
    \end{equation*}
    On the other hand, by BDG-inequality and,  we have
    \begin{equation*}
    \begin{aligned}
    \mathbb{E}\bigg[\sup_{0\leqslant t\leqslant \varepsilon}\big|\int_{0}^{t}\Delta\bar{X}_{s}C_{s}dW_{s}\big|^{2}\bigg]
    \leqslant 4\mathbb{E}\bigg[\big|\int_{0}^{\varepsilon}|\Delta\bar{X}_{s}C_{s}|^{2}ds\big|\bigg]\leq4\varepsilon K_{C}~\mathbb{E}\bigg[\sup_{0\leqslant t\leqslant \varepsilon}\big|\Delta\bar{X}_{t}\big|^{2}\bigg],
    \end{aligned}
    \end{equation*}
    where the positive constant $K_{C}$ is determined by the uniform bound of process $C$. Therefore, we can deduce that
    \begin{equation*}
    \begin{aligned}
    \mathbb{E}\bigg[\sup_{0\leqslant t\leqslant \varepsilon}\big|\Delta X_{t}\big|^{2}\bigg]&\leqslant 2\mathbb{E}\bigg[\sup_{0\leqslant t\leqslant \varepsilon}\big|\int_{0}^{t} A_{s}\Delta\bar{X}_{s}ds\big|^{2}\bigg]+2\mathbb{E}\bigg[\sup_{0\leqslant t\leqslant \varepsilon}\big|\int_{0}^{t}\Delta\bar{X}_{s}C_{s}dW_{s}\big|^{2}\bigg]\\
    &\leqslant\big(16\varepsilon~\|A\|^{2}_{\mathrm{BMO}_{2}}+8\varepsilon K_{C}\big)\mathbb{E}\bigg[\sup_{0\leqslant t\leqslant \varepsilon}\big|\Delta\bar{X}_{t}\big|^{2}\bigg].
    \end{aligned}
    \end{equation*}
    Choosing suitable $\varepsilon$ such that
    \begin{align*}
    0<16\varepsilon~\|A\|^{2}_{\mathrm{BMO}_{2}}+8\varepsilon K_{C}<1,
    \end{align*}
    we can get that $\Theta\big(\bar{X}\big):=X$ is a contraction map on $[0,\varepsilon]$. We can use the same method to verify
    \begin{align*}
    \mathbb{E}\bigg[\sup_{0\leqslant t\leqslant \varepsilon}\big|\Delta X_{t}\big|^{2}\bigg]<\mathbb{E}\bigg[\sup_{0\leqslant t\leqslant \varepsilon}\big|\Delta\bar{X}_{t}\big|^{2}\bigg]
    \end{align*}
    on $[\varepsilon,2\varepsilon],\cdots$, and finish our proof.
    \end{proof}
\section{Optimal pair for the robust control problem}\label{section 3}

In this section, we drive the optimal control and optimal value of the robust control problem \eqref{MMV1}, which involve the unique solutions to two BSDEs with unbounded coefficients. The maim technique is to find the saddle point. Following \cite{hu2023constrained}, we introduce a family of stochastic process process $R^{(\eta,u)}$ defined by
\begin{align*}
R_{t}^{(\eta,u)}:=h_{t}X_{t}^{u}+\frac{1}{2 \theta}\left(\Lambda_{t}^{\eta}Y_{t}-1\right),\quad t\in[0,T],~(\eta,u)\in\mathcal{A}\times\Pi,
\end{align*}
where $h$ and $Y$ are determined by two BSDEs studied in subsection \ref{section3.1}. To verify that $(\hat{\eta}$, $\hat{u})$ is the optimal pair of the MMV problem, we are aiming toprove the following five statements:
\begin{align}
&R_{T}^{(\eta,u)}=X_{T}^{u}+\frac{1}{2\theta}\left(\Lambda_{T}^{\eta}-1\right)\text{ for all }(\eta,u)\in\mathcal{A} \times\Pi;\label{RT}\\
&R_{0}^{(\eta,u)}=R_{0}\text{ is a constant for all }(\eta,u)\in \mathcal{A}\times\Pi;\label{R0}\\
&\mathbb{E}^{\mathbb{P}^{\hat{\eta}}}\left[R_{T}^{(\hat{\eta},u)}\right]\leqslant R_{0}\text{ for all }u\in\Pi;\label{eta hat statement}\\
&\mathbb{E}^{\mathbb{P}^{\eta}}\left[R_{T}^{(\eta,\hat{u})}\right]\geqslant R_{0}\text{ for all }\eta\in\mathcal{A};\label{u hat statement}\\
&\mathbb{E}^{\mathbb{P}^{\hat{\eta}}}\left[R_{T}^{(\hat{\eta},\hat{u})}\right]=R_{0}\label{eta u statement}.
\end{align}
\subsection{BSDEs with unbounded coefficients}\label{section3.1}
In this subsection, we will give the solvability of two BSDEs with unbounded coefficients, which plays a key role in solving the robust control problem.

First, we consider the following  backward stochastic differential equations:
\begin{equation}\label{h}
\left\{\begin{array}{l}
dh_{t}=\left[\left(-A_{t}+(D^{-1}_{t}B_{t})^{\prime}C_{t}\right)h_{t}+(D^{-1}_{t}B_{t})^{\prime}L_{t}+\frac{L_{t}^{\prime}L_{t}}{h_{t}}\right]dt+L_{t}^{\prime}dW_{t},\\
h_{T}=1.
\end{array}\right.
\end{equation}
It can be observed that equation \eqref{h} is a backward stochastic Riccati equation. Although this class of BSDEs with bounded coefficients are well studied in the literature, the solvability of \eqref{h} cannot be obtained from existing results since $A$ is unbounded.
\begin{lemma}\label{h solvability}
   Under assumption (A), the BSDE \eqref{h} admits a unique solution $(h, L) \in L^{\infty}_{\mathbb{F},\mathbb{P}}(0,T;\mathbb{R})\times L^{2,\mathrm{BMO}}_{\mathbb{F},\mathbb{P}}(0,T;\mathbb{R}^{n})$. In addition, the first component of the solution satisfies $h_{t}\geqslant\delta$, for all $t \in[0, T]$ and some $\delta > 0$. Moreover, there exists a constant $K>0$ such that for any $p>1$ and $t\in[0,T]$, we have
   \begin{equation*}
       \mathbb{E}_{t}\left[\mathcal{E} \left(\int_{t}^{T}-h^{-1}_{s}L^{\prime}_{s}dW_{s}\right)^{p}\right]\leq K.
   \end{equation*}
\end{lemma}
\begin{proof}
Under the assumption (A), it follows from BMO martingale theory that the process
\begin{equation*}
    N_{t}:= \mathcal{E}\left\{\int_0^t (-D^{-1}_{s}B_{s})^{\prime} d W_{s}\right\},
\end{equation*}
is a uniformly integrable martingale. Thus $\overline{W}_{t}:=W_{t}+\int_{0}^{t}D_{s}^{-1}B_{s}ds$ is a Brownian motion under probability $\overline{\mathbb{P}}$ defined by 
$$
\left.\frac{d \overline{\mathbb{P}}}{d \mathbb{P}}\right|_{\mathcal{F}_T}=N_{T}.
$$
 Denote by $\overline{\mathbb{E}}[\cdot]$ the expectation under $\overline{\mathbb{P}}$ and by $\overline{\mathbb{E}}_t[\cdot]$ the conditional expectation under $\overline{\mathbb{P}}$ given that $\mathcal{F}_t$. Set
\begin{equation*}
    \overline{h}_{T}=\mathrm{exp}\left(\int_0^T\left[-A_{t}+(D^{-1}_{t}B_{t})^{\prime}C_{t}\right] d t\right).
\end{equation*}
  Let $\varepsilon$ be a constant such that $0< \varepsilon< \frac{1}{\|A\|_{\mathrm{BMO}_{2}}}$. Then  $\|\varepsilon A\|_{\mathrm{BMO}_{2}}< 1$. Applying John-Nirenberg inequality (see \cite[Theorem 2.2]{BMO}) to BMO martingale  $\int_{0}^{\cdot}\varepsilon A_{s}dW_{s}$, we obtain for any $t\in[0,T]$ and $p>2$,
\begin{equation}\label{hT}
 \overline{\mathbb{E}}_t\left[K~\mathrm{exp}\left(p\int_t^T|A_{s}|d s\right)\right] \leqslant \overline{\mathbb{E}}_t\left[ K~\mathrm{exp}\left(\int_t^T\left(\varepsilon^{2}|A_{t}|^{2}+\frac{p^2}{4\varepsilon^{2}}\right) d t\right)\right]\leqslant K.
\end{equation}
Here and in the sequel, we shall use $K$ to represent a positive constant and can be different from line to line. Thus one can verify that $\overline{\mathbb{E}}_{t}\left[\overline{h}_{T}\right]$ is a square integrable martingale under $\overline{\mathbb{P}}$. By the martingale representation theorem, there exists $\overline{L} \in L^{2}_{\mathbb{F},\mathbb{\overline{P}}}(0,T;\mathbb{R}^{n})$ such that
\begin{equation}\label{bar h solution}
\overline{h}_{t}:=\overline{\mathbb{E}}_t\left[\mathrm{exp}\left(\int_0^T\left[-A_{t}+(D^{-1}_{t}B_{t})^{\prime}C_{t}\right] d t\right)\right] = \overline{h}_{0}+\int_{0}^{t}\overline{L}^{\prime}_{s}d\overline{W}_{s},
\end{equation}
which implies $\overline{h}_t>0$ for all $t\in[0,T]$. Denote \begin{equation}\label{1}
    h_{t}:= \frac{\exp\left(\int_{0}^{t}[-A_{s}+(D^{-1}_{s}B_{s})^{\prime}C_{s}]ds\right)}{\overline{h}_{t}},
\end{equation}
then by Jensen's inequality and \eqref{hT}, we have
\begin{equation*}
h_{t} =\frac{1}{\overline{\mathbb{E}}_t\left[\mathrm{exp}\left(\int_t^T\left[-A_{t}+(D^{-1}_{t}B_{t})^{\prime}C_{t}\right] d t\right)\right] } \leqslant \overline{\mathbb{E}}_t\left[\mathrm{exp}\left(\int_t^T\left[A_{t}-(D^{-1}_{t}B_{t})^{\prime}C_{t}\right] d t\right)\right] \leqslant K.
\end{equation*}
On the other hand, by \eqref{hT}, we obtain $h_{t}\geqslant \frac{1}{K}$. Thus, there exist a positive constant $\delta$, such that $h_{t}\geqslant \delta$, for all $t\in[0,T]$. In particular $h\in L^{\infty}_{\mathbb{F},\mathbb{P}}(0,T;\mathbb{R})$. Now denote
\begin{equation}\label{2}
    L_{t}:=- \frac{h_{t}\overline{L}_{t}}{\overline{h}_{t}},
\end{equation}
applying It\^o's formula to $h_{t}$, we obtain
\begin{equation}\label{h2}
 \begin{aligned}
d h_{t}= & \left[\left(- A_{t}+(D^{-1}_{t}B_{t})^{\prime}C_{t}\right) h_{t}+\frac{L^{\prime}_{t}L_{t}}{h_{t}}\right] d t+L^{\prime}_{t} d \overline{W}_{t} \\
= & \left[\left(- A_{t}+(D^{-1}_{t}B_{t})^{\prime}C_{t}\right) h_{t}+(D^{-1}_{t}B_{t})^{\prime}L_{t}+\frac{L^{\prime}_{t}L_{t}}{h_{t}}\right] d t+L^{\prime}_{t} d W_{t}.
\end{aligned}   
\end{equation}
Note that \eqref{h2} coincides with the BSDE \eqref{h}. From \eqref{1} and \eqref{2}, we see that $(h,L)$ has a one-to-one correspondence with $(\overline{h}, \overline{L})$. So the BSDE \eqref{h} also has a unique solution. Recalling the boundedness of $h$, from standard quadratic BSDE theory, we get $L\in L_{\mathbb{F},\mathbb{P}}^{2, \mathrm{BMO}}(0,T;\mathbb{R}^{n})$. 

Applying It\^o's formula to $\mathrm{ln} h$, we have
 \begin{equation*}
      d~\mathrm{ln}h  = [-A_{t}+(D_{t}^{-1}B_{t})^{\prime}C_{t}+(D_{t}^{-1}B_{t})^{\prime}h_{t}^{-1}L_{t}+\frac{1}{2}|h^{-1}_{t }L_{t}|^{2}]dt+ h^{-1}_{t}L^{\prime}_{t}dW_{t}.
    \end{equation*}  
 Therefore, 
    $$
        h_{T} = h_{t}~\cdot 
 \mathrm{exp}\left(\int_{t}^{T}[-A_{s}+(D_{s}^{-1}B_{s})^{\prime}C_{s}+(D_{s}^{-1}B_{s})^{\prime}h_{s}^{-1}L_{s}]ds\right)\cdot 
\mathrm{exp}\left(\int_{t}^{T}\frac{1}{2}|h^{-1}_{s}L_{s}|^{2} ds +\int_{t}^{T}h_{s}^{-1}L_{s}dW_{s}\right).
    $$
    Then,
    \begin{equation*}
       \mathcal{E} \left(\int_{t}^{T}-h^{-1}_{s}L^{\prime}_{s}dW_{s}\right) =\frac{h_{t}}{h_{T}}\mathrm{exp}\left(\int_{t}^{T}[-A_{s}+(D_{s}^{-1}B_{s})^{\prime}C_{s}+(D_{s}^{-1}B_{s})^{\prime}h_{s}^{-1}L_{s}]ds\right).
    \end{equation*}
    Let $\varepsilon$ be a constant such that $0< \varepsilon< \frac{1}{\|A\|_{\mathrm{BMO}_{2}}}$ and $0< \varepsilon< \frac{1}{\|L\|_{\mathrm{BMO}_{2}}}$. Then $\|\varepsilon A\|_{\mathrm{BMO}_{2}}< 1$ and $\|\varepsilon L\|_{\mathrm{BMO}_{2}}< 1$.
    For any $p>1$ and $t\in[0,T]$, we have
    \begin{equation}\label{hL}
        \begin{aligned}
        \mathbb{E}_{t}\left[\mathcal{E} \left(\int_{t}^{T}-h^{-1}_{s}L^{\prime}_{s}dW_{s}\right)^{p}\right] & = \mathbb{E}_{t}\left[ \left(\frac{h_{t}}{h_{T}}\right)^{p}\mathrm{exp}\left(p\int_{t}^{T}[-A_{s}+(D_{s}^{-1}B_{s})^{\prime}C_{s}+(D_{s}^{-1}B_{s})^{\prime}h_{s}^{-1}L_{s}]ds\right)\right]\\
        &\leqslant K\mathbb{E}_{t}\left[\mathrm{exp}\left(K\int_{t}^{T}(|A_{s}|+|L_{s}|)ds\right)\right]\\
        &\leqslant K \mathbb{E}_{t}\left[\mathrm{exp}\left(\int_{t}^{T}\left(\frac{\varepsilon^{2}A_{s}^{2}}{2}+\frac{\varepsilon^{2}|L_{s}|^{2}}{2}+\frac{K}{\varepsilon^{2}}\right)ds\right)\right]\\
        &\leqslant K\mathbb{E}_{t}\left[\mathrm{exp}\left(\int_{t}^{T}\varepsilon^{2}A_{s}^{2}ds\right)\right]^{\frac{1}{2}}\mathbb{E}_{t}\left[\mathrm{exp}\left(\int_{t}^{T}\varepsilon^{2}|L_{s}|^{2}ds\right)\right]^{\frac{1}{2}}\leq K.
\end{aligned}
    \end{equation}
 where $K$ denotes a constant that may vary from line to line; we used the boundedness of $h, B,C,D$ in the first inequality and John-Nirenberg Inequality to BMO-martingales $\int_{0}^{\cdot}\varepsilon A_{s}dW_{s}$ and  $\int_{0}^{\cdot}\varepsilon L_{s}dW_{s}$ in the last inequality. \\  
\end{proof}

Next, we introduce the following backward stochastic differential equation:
\begin{equation}\label{Y}
\left\{\begin{array}{l}
dY_{t}=\left[-|D_{t}^{-1}B_{t}+h^{-1}_{t}L_{t}|^{2}Y_{t}+2(D_{t}^{-1}B_{t}+h^{-1}_{t}L_{t})^{\prime}Z_{t}\right]dt+Z_{t}^{\prime}dW_{t},\\
Y_{T}=1,
\end{array}\right.
\end{equation}
where $(h,L)$ isthe unique solution of BSDE \eqref{h}. Since  $L$ is unbounded, the BSDE \eqref{Y} is a linear BSDE with unbounded coefficients and its solvability is by no means trivial. It is worth noting that  \cite{hu-nonhomogeneous-2023} solved a linear BSDE with unbounded coefficients. Unfortunately, our coefficients do not meet their conditions and thus we cannot apply their results to obtain the well-posedness of BSDE \eqref{Y}.
 \begin{lemma}\label{Y solvability}
   Under assumption (A), the BSDE \eqref{Y} admits a unique solution $(Y, Z) \in L^{\infty}_{\mathbb{F},\mathbb{P}}(0,T;\mathbb{R})\times L^{2,\mathrm{BMO}
    }_{\mathbb{F},\mathbb{P}}(0,T;\mathbb{R}^{n})$. In addition, the first component of the solution satisfies $Y_{t}\geqslant 1$, for all $t \in[0, T]$.
\end{lemma}
\begin{proof}
Let $(h,L)$ be the unique solution of BSDE \eqref{h}, we introduce the following two processes
\begin{equation*}
    J_{t} := \mathrm{exp}\left(\int_{0}^{t} |D^{-1}_{s}B_{s}+h^{-1}_{s}L_{s}|^{2}ds\right),
\end{equation*}
and
\begin{equation*}
M_{t}:= \mathcal{E}\left\{\int_0^t 2(D^{-1}_{s}B_{s}+h^{-1}_{s}L_{s})^{\prime} d W_{s}\right\}.
\end{equation*}
By Lemma \ref{h solvability}, $L\in L^{2,\mathrm{BMO}
    }_{\mathbb{F},\mathbb{P}}(0,T;\mathbb{R}^{n})$. Thus $M_{t}$ is a uniformly integrable martingale and $\widetilde{W}_{t}:=W_{t}+\int_{0}^{t}2(D_{s}^{-1}B_{s}+h^{-1}_{s}L_{s})ds$ is a Brownian motion under probability $\widetilde{\mathbb{P}}$ defined by 
$$
\left.\frac{d \widetilde{\mathbb{P}}}{d \mathbb{P}}\right|_{\mathcal{F}_T}=M_{T}.
$$

Let $\varepsilon$ be a constant such that $0< \varepsilon< \frac{1}{\|L\|_{\mathrm{BMO}_{2}}}$, then, for any $k>1$ and $p,q,r>1$ such that $\frac{1}{p}+\frac{1}{q}+\frac{1}{r}=1$, we have
\begin{equation}\label{Y bound}
\begin{aligned}
&\quad \mathbb{E}_{t}\left[\mathrm{exp}\left(-2k\int_{t}^{T}(D_{s}^{-1}B_{s}+h^{-1}_{s}L_{s})^{\prime}dW_{s}-k\int_{t}^{T}|D_{s}^{-1}B_{s}+h^{-1}_{s}L_{s}|^{2}ds\right)\right]\\
& = \mathbb{E}_{t}\left[\exp{\left(-2k\int_{t}^{T}h_{s}^{-1}L^{\prime}_{s}dW_{s}-k\int_{t}^{T}|h^{-1}_{s}L_{s}|^{2}ds\right)}\cdot \mathrm{exp}\left(-2k\int_{t}^{T}(D^{-1}_{s}B_{s})^{\prime}dW_{s}-k\int_{t}^{T}|D^{-1}_{s}B_{s}|^{2}ds\right)\right.\\
&\quad \quad \quad \left. \cdot \mathrm{exp}\left(-2k\int_{t}^{T}h^{-1}_{s}L^{\prime}_{s}D^{-1}_{s}B_{s}ds\right)\right]\\
& = \mathbb{E}_{t}\left[\mathcal{E}\left(\int_{t}^{T}-h^{-1}_{s}L^{\prime}_{s}dW_{s}\right)^{2k} \cdot\mathcal{E} \left(\int_{t}^{T}-(D^{-1}_{s}B_{s})^{\prime}dW_{s}\right)^{2k} \cdot \mathrm{exp}\left(-2k\int_{t}^{T}h^{-1}_{s}L^{\prime}_{s}D^{-1}_{s}B_{s}ds\right)\right]\\
& \leqslant \mathbb{E}_{t}\left[\mathcal{E} \left(\int_{t}^{T}-h^{-1}_{s}L^{\prime}_{s}dW_{s}\right)^{2kp}\right]^{\frac{1}{p}}\mathbb{E}_{t}\left[\mathcal{E} \left(\int_{t}^{T}-(D^{-1}_{s}B_{s})^{\prime}dW_{s}\right)^{2kq}\right]^{\frac{1}{q}}\mathbb{E}_{t}\left[\mathrm{exp}\left(\int_{t}^{T}\left(\varepsilon^{2}|L_{s}|^{2}+\frac{K}{\varepsilon^{2}}\right)ds\right)\right]^{\frac{1}{r}}\\
&\leqslant K,
\end{aligned}
\end{equation}
where we used \eqref{hL} and John-Nirenberg Inequality to BMO martingale  $\int_{0}^{\cdot}\varepsilon L_{s}dW_{s}$ in the last inequality.
Now, set
 \begin{equation*}
 \begin{aligned}
      Y_{t} = J_{t}^{-1}\widetilde{\mathbb{E}}_{t}\left[J_{T}\right]&= \widetilde{\mathbb{E}}_{t}\left[\mathrm{exp}\left(\int_{t}^{T}|D_{s}^{-1}B_{s}+h^{-1}_{s}L_{s}|^{2}ds\right)\right]\\
&=\mathbb{E}_{t}\left[\mathrm{exp}\left(-2\int_{t}^{T}(D_{s}^{-1}B_{s}+h^{-1}_{s}L_{s})^{\prime}dW_{s}-\int_{t}^{T}|D_{s}^{-1}B_{s}+h^{-1}_{s}L_{s}|^{2}ds\right)\right],
 \end{aligned}
\end{equation*}
 where we denote by $\widetilde{\mathbb{E}}_t[\cdot]$ the conditional expectation under $\widetilde{\mathbb{P}}$ given that $\mathcal{F}_t$. Then, by \eqref{Y bound}, $Y_{t}$ is bounded and $Y_{T}=1$. Similarly, we can prove $J_{t}Y_{t}$ is a square integrable martingale under $\widetilde{\mathbb{P}}$. By the martingale representation theorem, there exists $\widetilde{Z} \in L_{\mathbb{F},\mathbb{\widetilde{P}}}^2\left(0, T ; \mathbb{R}^n\right)$ such that
$$
J_{t}Y_{t}=J_{0} Y_{0}+\int_0^t \widetilde{Z}_{s}^{\prime} d \widetilde{W}_{s}.
$$
As a consequence,
$$
d(J_{t}Y_{t}) =\widetilde{Z}_{t}^{\prime} d \widetilde{W}_{t}  =J_{t}\left[Z_{t}^{\prime} d W_{t}+(D^{-1}_{t}B_{t}+h_{t}^{-1}L_{t})^{\prime}Z_{t}dt\right],
$$
where $Z_{t}=J_{t}^{-1} \widetilde{Z}_{t}$. By It\^o's lemma,
$$
d Y_{t} =d\left(J_{t}^{-1} \cdot J_{t} Y_{t}\right) =\left[-|D_{t}^{-1}B_{t}+h^{-1}_{t}L_{t}|^{2}Y_{t}+2(D_{t}^{-1}B_{t}+h^{-1}_{t}L_{t})^{\prime}Z_{t}\right]d t+Z_{t}^{\prime} d W_{t}.
$$
Thus $(Y, Z)$ satisfies \eqref{Y}.  Because $Y$ is bounded, we get
$$
\begin{aligned}
    \widetilde{\mathbb{E}}_\tau\left[\int_\tau^T|Z_{s}|^2 d s\right]  =\widetilde{\mathbb{E}}_\tau\left[\left(\int_\tau^T Z_{s}^{\prime} d \widetilde{W}_{s}\right)^2\right]  &=\widetilde{\mathbb{E}}_\tau\left[\left(Y(\tau)-1-\int_\tau^T|D_{s}^{-1}B_{s}+h_{s}^{-1}L_{s}|^{2} Y_{s} d s\right)^2\right] \\
& \leqslant K+K\widetilde{\mathbb{E}}_\tau\left[\left(\int_\tau^T|L_{s}|^{2} d s\right)^2\right] \leqslant K,
\end{aligned}
$$
for all stopping times $\tau \leqslant T$. Hence, $\int_0^{\cdot} Z_{s}^{\prime} d \widetilde{W}_{s}$ is a BMO martingale under $\widetilde{\mathbb{P}}$. Consequently $\int_0^{\cdot} Z_{s}^{\prime} d W_{s}$ is a BMO martingale under $\mathbb{P}$. This shows that $(Y, Z) \in L_{\mathbb{F},\mathbb{P}}^{\infty}(0, T ; \mathbb{R}) \times$ $L_{\mathbb{F},\mathbb{P}}^{2, \mathrm{BMO}}\left(0, T ; \mathbb{R}^n\right)$ is a solution of the  BSDE \eqref{Y}. Let us prove the uniqueness. Suppose 
$$
(Y, Z),\ (\hat{Y}, \hat{Z}) \in L_{\mathbb{F},\mathbb{P}}^{\infty}(0, T ; \mathbb{R}) \times L_{\mathbb{F},\mathbb{P}}^{2, \mathrm{BMO}}\left(0, T ; \mathbb{R}^n\right)
$$
are both solutions of \eqref{Y}. Set
$$
\Delta Y=Y-\hat{Y},\quad \Delta Z=Z-\hat{Z}.
$$
Then $(\Delta Y, \Delta Z)$ satisfies the following BSDE:
\begin{equation}\label{delta Y}
  \Delta Y_{t}=\int_t^T\left( |D_{s}^{-1}B_{s}+h_{s}^{-1}L_{s}|^{2}\Delta Y_{s}-2(D_{s}^{-1}B_{s}+h_{s}^{-1}L_{s})^{\prime} \Delta Z_{s}\right) d t-\int_t^T \Delta Z_{s}^{\prime} d W_{s} .  
\end{equation}
By the It\^o's formula, the solution of BSDE \eqref{delta Y} has the following representation:
$$
\Delta Y_{t}  = \Gamma_{t}^{-1}\mathbb{E}_{t}\left[\Gamma_{T} \Delta Y_{T}\right],
$$
where $\Gamma_{t} = 1+\int_{0}^{t}\Gamma_{s}\left[|D_{s}^{-1}B_{s}+h_{s}^{-1}L_{s}|^{2}ds-2(D_{s}^{-1}B_{s}+h_{s}^{-1}L_{s})^{\prime}dW_{s}\right]$. Since $\Delta Y_{T} = 0$, we have $\Delta Y_{t} = 0$ for all $t\in[0,T]$ and then $\Delta Z = 0$, for all $t\in[0,T]$. This completes the proof of the uniqueness.
\end{proof}
\subsection{Optimal pair}

In the following theorem, we drive the optimal control and optimal value of the robust control problem \eqref{MMV1}, which involve the unique solutions to the BSDEs \eqref{h} and \eqref{Y}. 
\begin{Theorem}\label{th-1}
The pair $(\hat{\eta},\hat{u})$ defined by 
\begin{equation}\label{optimal eta}
 \hat{\eta}=-D^{-1}B-h^{-1}L   
\end{equation}
and
\begin{equation}\label{optimal u}
\hat{u}=(hD^{\prime})^{-1}\left(\frac{\Lambda^{\eta}}{\theta}\big[(D^{-1}B+h^{-1}L)Y-Z\big]-X^{\hat{u}}L-hX^{\hat{u}}C\right)
\end{equation}
is optimal for the robust control problem \eqref{MMV1}, where $(h,L)$ and $(Y,Z)$ are the unique solutions of the BSDEs \eqref{h} and \eqref{Y}, respectively. Moreover,
\begin{equation*}
\sup_{u\in\Pi}\inf_{\eta\in\mathcal{A}} \mathbb{E}^{\mathbb{P}^{\eta}}\left[X_{T}^{u}+\frac{1}{2\theta}\left(\Lambda_{T}^{\eta}-1\right)\right]=xh_{0}+\frac{1}{2\theta}(Y_{0}-1).
\end{equation*}
\end{Theorem}
\begin{proof}
The whole proof will be splitted into two steps. For simplicity of notations, we denote:
\begin{equation}\label{notation}
    \phi_{t} := D^{-1}_{t}B_{t}+h_{t}^{-1}L_{t},\quad \alpha_{t}^{u}:=h_{t}D_{t}^{\prime}u_{t}+X_{t}^{u}L_{t}+h_{t}X_{t}^{u}C_{t}.
    \end{equation}
Apply It\^o's formula to $h_{t}X_{t}^{u}$, we can obtain that
\begin{equation}\label{hx}
    dh_{t}X_{t}^{u} = (\alpha_{t}^{u})^{\prime}\phi_{t}dt + (\alpha_{t}^{u})^{\prime}dW_{t}.
    \end{equation}
    \textbf{Step 1:}  To prove $(\hat{\eta},\hat{u})\in\mathcal{A}\times\Pi$. First, we show \eqref{optimal u} is well-defined, which means that when we use the feedback strategy \eqref{optimal u}, the state process equation
\begin{equation}\label{optimal x solvability}
    dX_{t}^{\hat{u}} = (A_{t}X_{t}^{\hat{u}}+\hat{u}^{\prime}_{t}B_{t})dt+(X_{t}^{\hat{u}}C_{t}^{\prime}+\hat{u}^{\prime}_{t}D_{t})dW_{t}.
\end{equation}
    admits a unique solution. Indeed, we can  confirm that 
\begin{equation}\label{optimal x}
X^{\hat{u}}_{t}=\frac{\theta h_{0}x+Y_{0}-\Lambda^{\hat{\eta}}_{t}Y_{t}}{\theta h_{t}},
\end{equation}
is a solution of \eqref{optimal x solvability}. Moreover, the above change $Y\rightarrow X^{\hat{u}}$ is invertible, so the uniqueness of \eqref{optimal x solvability} follows from that of \eqref{Y}.\\
Observe that
\begin{equation*}
    \alpha^{\hat{u}}_{t} = \frac{\Lambda_{t}^{\hat{\eta}}}{\theta}(\phi_{t} Y_{t} -Z_{t}).
\end{equation*}
By It\^o's formula, we have
\begin{align*}
d(\theta h_{t}X_{t}^{\hat{u}}+\Lambda_{t}^{\hat{\eta}}Y_{t})&=\bigg[\theta (\alpha_{t}^{\hat{u}})^{\prime}\phi_{t}+\frac{\Lambda_{t}^{\hat{\eta}}|Z_{t}|^{2}}{Y_{t}}-\frac{\Lambda_{t}^{\hat{\eta}}|\phi_{t}Y_{t}-Z_{t}|^{2}}{Y_{t}}+\Lambda_{t}^{\hat{\eta}}\hat{\eta}_{t}^{\prime}Z_{t}\bigg]dt\\
&\quad+\bigg[\theta(\alpha_{t}^{\hat{u}})^{\prime}+\Lambda_{t}^{\hat{\eta}}Z_{t}^{\prime}+\Lambda_{t}^{\hat{\eta}}Y_{t}\hat{\eta}_{t}^{\prime}\bigg]dW_{t}\\
&=\bigg[\Lambda_{t}^{\hat{\eta}}(\phi_{t}Y_{t}-Z_{t})^{\prime}\phi_{t}+\frac{\Lambda_{t}^{\hat{\eta}}(\phi_{t}Y_{t})^{\prime}(2Z_{t}-\phi_{t}Y_{t})}{Y_{t}}-\Lambda_{t}^{\hat{\eta}}\phi_{t}^{\prime}Z_{t}\bigg]dt\\
&\quad+\bigg[\Lambda_{t}^{\hat{\eta}}(\phi_{t}Y_{t}-Z_{t})^{\prime}+\Lambda_{t}^{\hat{\eta}}Z_{t}^{\prime}-\Lambda_{t}^{\hat{\eta}}Y_{t}\phi_{t}^{\prime}\bigg]dW_{t}\\
&=0,
\end{align*}
which implies for any $t\in[0,T]$, \eqref{optimal x} holds.\\
By It\^o's formula, we have
$$
\begin{aligned}
d(\Lambda_{t}^{\hat{\eta}})^{2}Y_{t}&=(\Lambda_{t}^{\hat{\eta}})^{2}\bigg[\frac{|Z_{t}|^{2}}{Y_{t}}-\frac{|\phi_{t}Y_{t}-Z_{t}|^{2}}{Y_{t}}+2\hat{\eta}_{t}^{\prime}Z_{t}+Y_{t}|\hat{\eta}_{t}|^{2}\bigg]dt+(\Lambda_{t}^{\hat{\eta}})^{2}\bigg[Z_{t}^{\prime}+2Y_{t}\hat{\eta}_{t}^{\prime}\bigg]dW_{t}\\
&=(\Lambda_{t}^{\hat{\eta}})^{2}\bigg[\frac{(\phi_{t}Y_{t})^{\prime}(2Z_{t}-\phi_{t}Y_{t})}{Y_{t}}-2\phi_{t}^{\prime}Z_{t}+Y_{t}|\phi_{t}|^{2}\bigg]dt+(\Lambda_{t}^{\hat{\eta}})^{2}\bigg[Z_{t}^{\prime}-2Y_{t}\phi_{t}^{\prime}\bigg]dW_{t}\\
&=(\Lambda_{t}^{\hat{\eta}})^{2}Y_{t}\cdot\frac{Z_{t}^{\prime}-2Y_{t}\phi_{t}^{\prime}}{Y_{t}}dW_{t},
\end{aligned}
$$
which implies
\begin{align*}
d\frac{(\theta h_{t}X_{t}^{\hat{u}}-\theta h_{0}x-Y_{0})^{2}}{Y_{t}}=\frac{(\theta h_{t}X_{t}^{\hat{u}}-\theta h_{0}x-Y_{0})^{2}}{Y_{t}}\cdot\frac{Z_{t}^{\prime}-2Y_{t}\phi_{t}^{\prime}}{Y_{t}}dW_{t}.
\end{align*}
Therefore
\begin{align*}
\frac{(\theta h_{t}X_{t}^{\hat{u}}-\theta h_{0}x-Y_{0})^{2}}{Y_{t}}
\end{align*}
is a local martingale, which means there exists an increasing sequence of stopping times $\tau_{n}\uparrow+\infty$ such that when $n\rightarrow+\infty$, it holds for any stopping time $\iota\leqslant T$ that
\begin{align*}
\mathbb{E}\left[\frac{(\theta h_{\iota\wedge\tau_{n}}X_{\iota\wedge\tau_{n}}^{\hat{u}}-\theta h_{0}x-Y_{0})^{2}}{Y_{\iota\wedge \tau_{n}}}\right]=Y_{0}.
\end{align*}
Let $n\rightarrow\infty$, by Fatou's lemma we have
\begin{align*}
\mathbb{E}\left[\frac{(\theta h_{\iota}X_{\iota}^{\hat{u}}-\theta h_{0}x-Y_{0})^{2}}{Y_{\iota}}\right]\leqslant Y_{0}.
\end{align*}
Since $Y,h$ are bounded, we get
\begin{align}\label{ Xexpectation bound}
\mathbb{E}\left[(X_{\iota}^{\hat{u}})^{2}\right]<C
\end{align}
for some constant $C$ and for any stopping time $\iota\leqslant T$. From \cite[Lemma 4.3]{Huetal.2022}, we can obtain $\hat{u}\in\Pi$.
Moreover, by \eqref{optimal x}, we get
$$
\Lambda_{t}^{\hat{\eta}} = \frac{\theta h_{0}x+Y_{0}-\theta h_{t}X_{t}^{\hat{u}}}{Y_{t}}.
$$
Combined with \eqref{ Xexpectation bound} and the fact that $Y,h$ are bounded, we immediately obtain for any $t\in[0,T]$,
\begin{align*}
\mathbb{E}[(\Lambda_{t}^{\hat{\eta}})^{2}]<\infty,
\end{align*}
which proves $\hat{\eta}\in \mathcal{A}$.\\
\textbf{Step 2.} To prove $(\hat{\eta},\hat{u}) $ is the optimal pair of the MMV problem, we only need to verify \eqref{RT}-\eqref{eta u statement}. First, from the definition of processes of $(h,L)$ and $(Y,Z)$, \eqref{RT} and \eqref{R0} are satisfied. Applying It\^o's formula to $R_{t}^{(\eta,u)}$ yields 
\begin{equation}\label{R}
\begin{aligned}
dR_{t}^{(\eta,u)}&=\bigg[\frac{1}{2\theta}\Lambda_{t}Y_{t}|\eta_{t}|^{2}+\eta_{t}^{\prime}\left(\frac{1}{\theta}\Lambda_{t} Z_{t}+\alpha_{t}^{u}\right)+(\alpha_{t}^{u})^{\prime}\phi_{t}+\frac{1}{2\theta}\Lambda_{t}\left(\frac{|Z_{t}|^{2}}{Y_{t}}-\frac{|\phi_{t}Y_{t}-Z_{t}|^{2}}{Y_{t}}\right)\bigg]dt\\
&\quad+\left[(\alpha_{t}^{u})^{\prime}+\frac{1}{2\theta}\Lambda_{t}Z_{t}^{\prime}+\frac{1}{2\theta}\Lambda_{t}Y_{t}\eta_{t}^{\prime}\right]dW_{t}^{\eta}\\
&=\frac{\Lambda_{t}}{2\theta}\bigg[Y_{t}\left|\eta_{t}+\frac{1}{Y_{t}}\left(Z_{t}+\frac{\theta \alpha_{t}^{u}}{\Lambda_{t}}\right)\right|^{2}-\bigg(\frac{|\phi_{t}Y_{t}-Z_{t}|^{2}}{Y_{t}}-\frac{|Z_{t}|^{2}}{Y_{t}}+\frac{1}{Y_{t}}\left|Z_{t}+\frac{\theta \alpha^{u}_{t}}{\Lambda_{t}}\right|^{2}-\frac{2\theta (\alpha_{t}^{u})^{\prime}\phi_{t}}{\Lambda_{t}}\bigg)\bigg]dt\\
&\quad+\left[(\alpha_{t}^{u})^{\prime}+\frac{1}{2\theta}\Lambda_{t}Z_{t}^{\prime}+\frac{1}{2\theta}\Lambda_{t}Y_{t}\eta_{t}^{\prime}\right]dW_{t}^{\eta}.
\end{aligned}
\end{equation}
Let $\hat{\eta}$ be defined in \eqref{optimal eta}. For any $u\in \Pi$ and an increasing sequence of stopping times $\tau_{n}\uparrow+\infty$, we have
\begin{equation}\label{eta hat}
\begin{aligned}
\mathbb{E}^{\mathbb{P}^{\hat{\eta}}}\left[R_{t\wedge\tau_{n}}^{(\hat{\eta},u)}\right]&=xh_{0}+\frac{1}{2\theta}(Y_{0}-1)+\mathbb{E}^{\mathbb{P}^{\hat{\eta}}}\int_{0}^{t\wedge\tau_{n}}\bigg[\frac{1}{2\theta}\Lambda_{s}^{\hat{\eta}}Y_{s}|(\alpha_{s}^{u})^{\prime}|^{2}+\hat{\eta}_{s}^{\prime}\left(\frac{1}{\theta}\Lambda_{s}^{\hat{\eta}}Z_{s}+\alpha_{s}^{u}\right)+(\alpha_{s}^{u})^{\prime}\phi_{s}\\
&\qquad\qquad+\frac{1}{2\theta}\Lambda_{s}^{\hat{\eta}}\left(\frac{|Z_{s}|^{2}}{Y_{s}}-\frac{|\phi_{s}Y_{s}-Z_{s}|^{2}}{Y_{s}}\right)\bigg]ds\\
&=xh_{0}+\frac{1}{2\theta}(Y_{0}-1)+\mathbb{E}^{\mathbb{P}^{\hat{\eta}}}\int_{0}^{t\wedge\tau_{n}}\frac{1}{2\theta}\Lambda_{s}^{\hat{\eta}}\bigg[Y_{s}|\phi_{s}|^{2}-2\phi_{s}^{\prime}Z_{s}+\phi_{s}^{\prime}(2Z_{s}-\phi_{s}Y_{s})\bigg]ds\\
&=xh_{0}+\frac{1}{2\theta}(Y_{0}-1).
\end{aligned}    
\end{equation}
For any $u \in \Pi$, in Theorem \ref{state space theorem}, we have verified that $\mathbb{E}\left[\sup _{t \in[0, T]}\left(X_t^u\right)^2\right]<\infty$. Since $h$ and $Y$ are bounded, we get
\begin{equation}\label{local}
\begin{aligned}
\mathbb{E}\left[\sup _{t \in[0, T]}\left|\Lambda_t^{\hat{\eta}} R_t^{\hat{\eta}, u}\right|\right] & \leqslant K \mathbb{E}\left[\sup _{t \in[0, T]}\left|\Lambda_t^{\hat{\eta}} X_t^u\right|\right]+K \mathbb{E}\left[\sup _{t \in[0, T]}\left(\Lambda_t^{\hat{\eta}}\right)^2\right]+K \mathbb{E}\left[\sup _{t \in[0, T]} \Lambda_t^{\hat{\eta}}\right] \\
& \leqslant K \mathbb{E}\left[\sup _{t \in[0, T]}\left(X_t^u\right)^2\right]+K \mathbb{E}\left[\sup _{t \in[0, T]}\left(\Lambda_t^{\hat{\eta}}\right)^2\right]+K \\
& \leqslant K \mathbb{E}\left[\sup _{t \in[0, T]}\left(X_t^u\right)^2\right]+K\mathbb{E}\left[\left(\Lambda_T^{\hat{\eta}}\right)^2\right]+K<\infty
\end{aligned}
\end{equation}
where $K>0$ denotes a constant that may vary from line to line and we use the elementary inequality $2 a b \leqslant a^2+b^2$ in the second inequality, Doob's inequality in the third inequality. Sending $n \rightarrow \infty$ in \eqref{eta hat} and using the dominated convergence theorem, we get \eqref{eta hat statement}.

Next, let $\hat{u}$ be defined in \eqref{optimal u}. For any $\eta \in \mathcal{A}$ and an increasing sequence of stopping times $\tau_{n}\uparrow+\infty$, we have
\begin{equation}\label{local-2}
\begin{aligned}
\mathbb{E}^{\mathbb{P}^{\eta}}\left[R_{t\wedge\tau_{n}}^{(\eta,\hat{u})}\right]&=xh_{0}+\frac{1}{2\theta}(Y_{0}-1)+\mathbb{E}^{\mathbb{P}^{\eta}}\int_{0}^{t\wedge\tau_{n}}\frac{\Lambda_{s}}{2\theta}\bigg[Y_{s}\left|\eta_{s}+\frac{1}{Y_{s}}\left(Z_{s}+\frac{\theta\alpha^{\hat{u}}_{s}}{\Lambda_{s}}\right)\right|^{2}\\
&\qquad\qquad-\bigg(\frac{|\phi_{s}Y_{s}-Z_{s}|^{2}}{Y_{s}}-\frac{|Z_{s}|^{2}}{Y_{s}}+\frac{1}{Y_{s}}\left|Z_{s}+\frac{\theta\alpha^{\hat{u}}_{s}}{\Lambda_{s}}\right|^{2}-\frac{2\theta(\alpha^{\hat{u}}_{s})^{\prime}\phi_{s}}{\Lambda_{s}}\bigg)\bigg]ds\\
&=xh_{0}+\frac{1}{2\theta}(Y_{0}-1)+\mathbb{E}^{\mathbb{P}^{\eta}}\int_{0}^{t\wedge\tau_{n}}\frac{\Lambda_{s}}{2\theta}\bigg[Y_{s}\left|\eta_{s}+\frac{1}{Y_{s}}\left(Z_{s}+\frac{\theta\alpha^{\hat{u}}_{s}}{\Lambda_{s}}\right)\right|^{2}\\
&\qquad\qquad-\bigg(\phi_{s}^{\prime}(\phi_{s}Y_{s}-2Z_{s})+|\phi_{s}|^{2}Y_{s}-2(\phi_{s}Y_{s}-Z_{s})^{\prime}\phi_{s}\bigg)\bigg]ds\\
&=xh_{0}+\frac{1}{2\theta}(Y_{0}-1)+\mathbb{E}^{\mathbb{P}^{\eta}}\int_{0}^{t\wedge\tau_{n}}\frac{\Lambda_{s}}{2\theta}\bigg[Y_{s}\left|\eta_{s}+\frac{1}{Y_{s}}\left(Z_{s}+\frac{\theta\alpha^{\hat{u}}_{s}}{\Lambda_{s}}\right)\right|^{2}\bigg]ds\\
&\geqslant xh_{0}+\frac{1}{2\theta}(Y_{0}-1).
\end{aligned}
\end{equation}
Similarly, we can also prove $\mathbb{E}\left[\sup _{t \in[0, T]}\left|\Lambda_t^\eta R_t^{\eta, \hat{u}}\right|\right]<\infty$ and obtain \eqref{u hat statement}.
 
Finally, let $(\hat{\eta},\hat{u})$ be defined in \eqref{optimal eta} and \eqref{optimal u} and from Step 1, we have
\begin{align*}
\mathbb{E}^{\mathbb{P}^{\hat{\eta}}}\left[R_{t\wedge\tau_{n}}^{(\hat{\eta},\hat{u})}\right]=xh_{0}+\frac{1}{2\theta}(Y_{0}-1).
\end{align*}
This completes the proof.
\end{proof}
\section{Connection with the control problem with mean-variance cost functional}\label{section 4}
In this section, we study the following stochastic control problem:
\begin{equation}\label{LQ}
    \underset{u \in \Pi}{\operatorname{sup}}\left[\mathbb{E}[X^u_{T}]-\frac{\theta}{2}\mathrm{Var}(X^u_{T})\right],
\end{equation} 
where the state process satisfies \eqref{SDE}. We can show that the optimal control $\hat{u}$ of the robust stochastic control problem \eqref{MMV1} is exactly the optimal control of problem \eqref{LQ}.
\begin{Theorem}
   Under assumption (A), the control $\hat{u}$ defined by \eqref{optimal u} is the optimal control of stochastic control problem \eqref{LQ}. Moreover, the stochastic control problem \eqref{LQ} has the same optimal value with robust stochastic control problem \eqref{MMV1},
   \begin{equation*}
\underset{{u\in \Pi}}{\operatorname{sup}}\left[\mathbb{E}[X_{T}]-\frac{\theta}{2}\mathrm{Var}(X_{T})\right] = h_{0}x+\frac{1}{2\theta}(Y_{0}-1).
   \end{equation*}
\end{Theorem}
\begin{proof}
First, let us consider 
\begin{equation}\label{F(K)}
    F(K):= \underset{u\in \Pi^{K}}{\operatorname{inf}}\mathbb{E}\left[(X_{T}-K)^{2}\right],
\end{equation}
where
$$
\Pi^{K}:=\left\{u\in \Pi\mid \mathbb{E}[X_{T}^{u}] = K\right\}.
$$
It follows from \cite{trybula2019continuous}, that
\begin{equation}\label{connection}
    \begin{aligned}
        \underset{{u\in \Pi}}{\operatorname{sup}}\left[\mathbb{E}[X_{T}]-\frac{\theta}{2}\mathrm{Var}(X_{T})\right] &=  \underset{{K\in \mathbb{R}}}{\operatorname{sup}} \underset{{u\in \Pi^{K}}}{\operatorname{sup}}\left[K-\frac{\theta}{2}\mathbb{E}\left[(X_{T}-K)^{2}\right]\right]\\
        & = \underset{{K\in \mathbb{R}}}{\operatorname{sup}} \left[K-\frac{\theta}{2}\underset{{u\in \Pi^{K}}}{\operatorname{inf}}\mathbb{E}\left[(X_{T}-K)^{2}\right]\right]\\
        & = \underset{{K\in \mathbb{R}}}{\operatorname{sup}} \left[K-\frac{\theta}{2}F(K)\right].
    \end{aligned}
\end{equation}
For solving problem  \eqref{F(K)}, we introduce the following problem:
\begin{equation}\label{J}
    J(K,\gamma):=\inf _{u \in \Pi}\left[\mathbb{E}\left(X_T-\gamma\right)^2-(K-\gamma)^2\right].
\end{equation}
It follows the Lagrangian duality theorem (see \cite{MR0238472}), that
\begin{equation}\label{F(K)2}
F(K)=\sup_{\gamma\in \mathbb{R}}J(K,\gamma).
\end{equation}
Now, we first solve the problem \eqref{J} and we also use the notations defined by \eqref{notation}. Applying It\^o's formula to $Y_{t}^{-1}(h_{t}X_{t}^{u}-\gamma)^{2}$ gives
\begin{equation}\label{hx-k}
     \begin{aligned}
d Y_{t}^{-1}(h_{t}X_{t}^{u}-\gamma)^2 &=\left[ Y_{t}^{-1}\left\{\alpha^{u}_{t}+\left(\phi_{t}-\frac{Z_{t}}{Y_{t}}\right)(h_{t}X_{t}^{u}-\gamma)\right\}^{\prime}  \times\left\{ \alpha^{u}_{t}+\left(\phi_{t}-\frac{Z_{t}}{Y_{t}}\right)(h_{t}X^{u}_{t}-\gamma)\right\} \right]d t \\
&\quad+\left[2 Y_{t}^{-1}(h_{t}X_{t}^{u}-\gamma)(\alpha^{u}_{t})^{\prime}-(h_{t}X_{t}^{u}-\gamma)^2 (Y^{-2}_{t}Z_{t})^{\prime}\right] d W_{t} .
\end{aligned}   
\end{equation}
Notice that $(Y,Z)\in L^{\infty}_{\mathbb{F},\mathbb{P}}(0,T;\mathbb{R})\times L^{2,\mathrm{BMO}
}_{\mathbb{F},\mathbb{P}}(0,T;\mathbb{R}^{n})$ and $X\in L^{2}_{\mathbb{F}}(\Omega;C(0,T);\mathbb{R})$. We have that in \eqref{hx-k} the integrand of the stochastic integral with respect to the Brownian motion is locally square-integrable under $\mathbb{P}$. So this stochastic integral is an $(\mathbb{F}, \mathbb{P})$-local martingale. Then there is an increasing sequence of stopping times $\left\{\tau_n\right\}_{n\geq 1}$ such that $\tau_k \uparrow T$ as $n \rightarrow \infty$ and the local martingale stopped by $\left\{\tau_n\right\}_{n\geq 1}$ is a true $(\mathbb{F}, \mathbb{P})$ martingale. So integrating from 0 to $\tau_n$ and taking expectations on both sides of \eqref{hx-k} yield
\begin{equation}\label{E hx-k}
\begin{aligned}
&\mathbb{E}\left[Y_{\tau_{n}}^{-1}(h_{\tau_{n}}X_{\tau_{n}}^{u}-\gamma)^2\right] \\
&= Y_{0}^{-1}(h_{0}x-\gamma)^2+\mathbb{E}\left[\int_{0}^{\tau_{n}} Y_{t}^{-1}\left\{\alpha^{u}_{t}+\left(\phi_{t}-\frac{Z_{t}}{Y_{t}}\right)(h_{t}X_{t}^{u}-\gamma)\right\}^{\prime}  \times\left\{ \alpha^{u}_{t}+\left(\phi_{t}-\frac{Z_{t}}{Y_{t}}\right)(h_{t}X_{t}^{u}-\gamma)\right\} d t\right].
\end{aligned}   
\end{equation}
Since $X\in L_{\mathbb{F},\mathbb{P}}^2(\Omega;C([0,T];\mathbb{R}))$ and $Y,h$ is bounded, we have that the family $\left\{Y_{\tau_{n}}^{-1}\left(h_{\tau_{n}}X_{\tau_{n}}-\gamma\right)^2\right\}_{n\geq 1}$ is  dominated by a nonnegative, integrable random variable. In addition, we note that the integrand on the right-hand side of \eqref{hx-k} is nonnegative. Sending $n$ to $\infty$ and using the dominated convergence theorem and the monotone convergence theorem to the left-hand side and the right-hand side of \eqref{E hx-k}, respectively, yield
\begin{equation}\label{XT-gamma2}
\begin{aligned}
&\mathbb{E}[(X_{T}^{u}-\gamma)^{2}]\\
&=Y_{0}^{-1}(h_{0}x-\gamma)^{2}+\mathbb{E}\left[\int_{0}^{T} Y_{t}^{-1}\left\{\alpha^{u}_{t}+\left(\phi_{t}-\frac{Z_{t}}{Y_{t}}\right)(h_{t}X_{t}^{u}-\gamma)\right\}^{\prime}  \times\left\{ \alpha^{u}_{t}+\left(\phi_{t}-\frac{Z_{t}}{Y_{t}}\right)(h_{t}X_{t}^{u}-\gamma)\right\} d t\right].
\end{aligned}
\end{equation}
Therefore, by \eqref{notation}, we get the optimal feedback strategy  
\begin{equation}\label{feedback stategy}
h_{t}D_{t}^{\prime} \hat{u}_{t}^{\gamma(K)}+X_{t}^{\hat{u}^{\gamma(K)}}L_{t}+h_{t}X_{t}^{\hat{u}^{\gamma(K)}}C_{t}  = -\left(\phi_{t}-\frac{Z_{t}}{Y_{t}}\right)\left(h_{t}X_{t}^{\hat{u}^{\gamma(K)}}-\gamma(K)\right),
\end{equation}
and
\begin{equation*}
    J(K,\gamma) = Y_{0}^{-1}(h_{0}x-\gamma)^{2}.
\end{equation*}
Then, combined with the fact  $Y_{0}> 1$ obtained by Lemma \ref{Y solvability}, we get
\begin{equation*}
     F(K)=\sup_{\gamma\in \mathbb{R}}\left(Y_{0}^{-1}(h_{0}x-\gamma)^{2}-(K-\gamma)^{2}\right)= \frac{Y_{0}^{-1}(K-h_{0}x)^{2}}{1-Y_{0}^{-1}},
\end{equation*}
and 
\begin{equation*}
\hat{\gamma}(K) = \frac{Y_{0}^{-1}h_{0}x-K}{Y_{0}^{-1}-1}.    
\end{equation*}
Therefore, by \eqref{connection}, it holds
\begin{equation*}
    \hat{K} = h_{0}x+\frac{1-Y_{0}^{-1}}{\theta Y_{0}^{-1}},\quad \underset{{u\in \Pi}}{\operatorname{sup}}\left[\mathbb{E}[X_{T}]-\frac{\theta}{2}\mathrm{Var}(X_{T})\right] = h_{0}x+\frac{1}{2\theta}(Y_{0}-1),
\end{equation*}
and
\begin{equation}\label{hat gamma}
    \hat{\gamma}(\hat{K}) = h_{0}x+\frac{Y_{0}}{\theta}.
\end{equation}
Substituting \eqref{hat gamma} into \eqref{feedback stategy} and  using \eqref{optimal x} yield
\begin{equation*}
\begin{aligned}
\hat{u}_{t}
& = (h_{t}D^{\prime}_{t})^{-1}\left[-(\phi_{t}-\frac{Z_{t}}{Y_{t}})\left(h_{t}X_{t}^{\hat{u}}-h_{0}x-\frac{Y_{0}}{\theta}\right)-X_{t}^{\hat{u}}L_{t}-h_{t}X_{t}^{\hat{u}}C_{t}\right]\\
& = (h_{t}D_{t}^{\prime})^{-1}\left[(\phi_{t}Y_{t}-Z_{t})\left(\frac{\theta h_{0}x+Y_{0}-\theta h_{t}X_{t}^{\hat{u}}}{\theta Y_{t}}\right)-X_{t}^{\hat{u}}L_{t}-h_{t}X^{\hat{u}}_{t}C_{t}\right]\\
& = (h_{t}D_{t}^{\prime})^{-1}\left[\frac{\Lambda^{\hat{\eta}}}{\theta}\big(\phi_{t}Y_{t}-Z_{t}\big)-X_{t}^{\hat{u}}L_{t}-h_{t}X^{\hat{u}}_{t}C_{t}\right],
\end{aligned}
\end{equation*}
which is exactly the optimal strategy of the robust stochastic control problem. Moreover, from the Theorem \ref{th-1}, we obtain $\hat{u}\in \Pi$. This completes the proof.
\end{proof}
\section{Some applications}\label{section 5}
\subsection{Application to MMV and MV portfolio selection problem}

In this subsection, we consider a financial market consisting of a risk-free asset (the money market instrument or bond) whose price is $S_{0}$ and $n$ risky securities (the stocks) whose prices are $S_{1}, \ldots, S_{n}$, i.e., the number of risky securities is equal to the dimension of the Brownian motion. The asset prices $S_{k}, k=0,1, \ldots, n$, are driven by SDEs:
\begin{align*}
\left\{\begin{array}{l}
dS_{0,t}=r_{t}S_{0,t}dt\\
S_{0,0}=s_{0},
\end{array}\right.
\end{align*}
and
\begin{align*}
\left\{\begin{array}{l}
dS_{k,t}=S_{k,t}\left(\left(\mu_{k,t}+r_{t}\right)dt+\sum_{j=1}^{n}\sigma_{kj,t}dW_{j,t}\right),\\
S_{k,0}=s_{k},
\end{array}\right.
\end{align*}
where $r$ is the interest rate process andfor every $k=1,\ldots,n$, $\mu_{k}$ and $\sigma_{k}:=\left(\sigma_{k 1}, \ldots, \sigma_{k n}\right)$ are the mean excess return rate process and volatility rate process of the $k$ th risky security.

Define the mean excess return vector $\mu=\left(\mu_{1},\ldots,\mu_{n}\right)^{\prime}$ and volatility matrix
\begin{align*}
\sigma=\left(\begin{array}{c}
\sigma_{1} \\
\vdots \\
\sigma_{n}
\end{array}\right) \equiv\left(\sigma_{k j}\right)_{n\times n}.
\end{align*}

Consider a small investor whose actions cannot affect the asset prices. The investor will decide at every time $t \in[0, T]$ the amount $\pi_{j, t}$ of wealth to invest in the $j$ th risky asset, $j=1,\ldots,n$. The vector process $\pi:=\left(\pi_{1}, \ldots, \pi_{n}\right)^{\prime}$ is called a portfolio of the investor. Then the investor's
self-financing wealth process $X$ corresponding to a portfolio $\pi$ is the unique strong solution of the SDE:
\begin{equation*}
\left\{\begin{array}{l}
dX_{t}=\left(r_{t}X_{t}+\pi_{t}^{\prime}\mu_{t}\right)dt+\pi_{t}^{\prime}\sigma_{t}dW_{t},\\
X_{0}=x.
\end{array}\right.
\end{equation*}
We assume the following:\\
\textbf{Assumption (B).}
\begin{equation}\label{assumption}
r\in L_{\mathbb{F},\mathbb{P}}^{2,\mathrm{BMO}}\left(0,T;\mathbb{R}\right),\quad \mu\in L_{\mathbb{F},\mathbb{P}}^{\infty}\left(0,T;\mathbb{R}^{n}\right),\quad \sigma\in L_{\mathbb{F},\mathbb{P}}^{\infty}\left(0,T;\mathbb{R}^{n\times n}\right),
\end{equation}
and there exists a constant $\delta>0$ such that $\sigma \sigma^{\prime}\geqslant\delta I_{n\times n}$ for $t\in[0,T]$.

We consider the following MMV problem and MV problem:
\begin{equation}\label{MMV-2}
\sup_{\pi\in\Pi}\inf_{\eta\in\mathcal{A}} \mathbb{E}^{\mathbb{P}^{\eta}}\left[X_{T}^{\pi}+\frac{1}{2\theta}\left(\Lambda_{T}^{\eta}-1\right)\right],
\end{equation}
and
\begin{equation}\label{MV}
    \underset{\pi \in \Pi}{\operatorname{sup}}\left[\mathbb{E}[X_{T}]-\frac{\theta}{2}\mathrm{Var}(X_{T})\right],
\end{equation}
where $\theta$ is a given positive constant. By the results obtained in the previous sections, we can easily obtain the optimal strategies of the MMV and MV problems respectively and show that the optimal strategies to MMV and MMV problems coincide.
\begin{corollary}
Suppose assumption (B) holds. The pair $(\hat{\eta}, \hat{\pi})$ defined by
$$
\hat{\eta}=-\sigma^{-1} \mu-h^{-1} L
$$
and
$$
\hat{\pi}=\left(h \sigma^{\prime}\right)^{-1}\left(\frac{\Lambda^{\hat{\eta}}}{\theta}\left[\left(\sigma^{-1} \mu+h^{-1} L\right) Y-Z\right]-X^{\hat{\pi}} L\right)
$$
is optimal for the MMV problem \eqref{MMV-2} and $\hat{\pi}$ is also the optimal control of MV problem \eqref{MV}, where $(h, L)$ and $(Y, Z)$ are the unique solutions of following BSDEs, respectively,
\begin{equation}
\left\{\begin{array}{l}
dh_{t}=\left[-r_{t}h_{t}+(\sigma^{-1}_{t}\mu_{t})^{\prime}L_{t}+\frac{L_{t}^{\prime}L_{t}}{h_{t}}\right]dt+L_{t}^{\prime}dW_{t},\\
h_{T}=1,
\end{array}\right.
\end{equation}
and
\begin{equation}
\left\{\begin{array}{l}
dY_{t}=\left[-|\sigma_{t}^{-1}\mu_{t}+h_{t}^{-1}L_{t}|^{2}Y_{t}+2(\sigma_{t}^{-1}\mu_{t}+h^{-1}_{t}L_{t})^{\prime}Z_{t}\right]dt+Z_{t}^{\prime}dW_{t},\\
Y_{T}=1.
\end{array}\right.
\end{equation}
Moreover, the problem \eqref{MMV-2} and \eqref{MV} have the same optimal value
$$
h_{0}x+\frac{1}{2\theta}(Y_{0}-1).
$$
\end{corollary}
\begin{remark}
    For the MMV portfolio selection problem, our result partially answers the question proposed in \cite{hu2023constrained} for random (and unbounded) interest rate and complete market. For the MV portfolio selection problem, by assuming the appreciation rate and the volatility is bounded, we allow for unbounded random interest rate. This complements the mean-variance portfolio literature (e.g. Shen \cite{shenyang2015}, etc.).
\end{remark}
\subsection{Application to MMV and MV investment-reinsurance problem}

In this subsection, we slightly modify the setup of former subsection. Let $(\Omega,\mathcal{F},\mathbb{P})$ be a probability space on which defined a standard $n$-dimensional Brownian motion. We further assume the Poisson random measure $\bar{\gamma}$ on $[0, T]\times\Omega\times \mathbb{R}_{+}$ is independent with the Brownian motion $W$ under the probability measure $\mathbb{P}$. We denote the intensity $\lambda>0$ and the distribution function $\nu:\mathbb{R}^{+}\rightarrow[0,1]$, and the compensated Poisson random measure can be write as
\begin{equation*}
\tilde{\gamma}(dt,dy)=\bar{\gamma}(dt,dy)-\lambda\nu(dy)dt.
\end{equation*}
Define the filtration $\mathbb{F}=\{\mathcal{F}_t\}_{t\geq 0}$ as the augmented natural filtration generated by $W$ and $\bar{\gamma}$. Define the filtration $\mathbb{F}^{W}=\{\mathcal{F}^W_t\}_{t\geq 0}$ as the augmented natural filtration generated by $W$. 

Let $\mathcal{B}(\mathbb{R}_{+})$ be the Borel $\sigma$-algebra of $\mathbb{R}_{+}$, $\mathcal{P}$ the $\mathbb{F}$-predictable $\sigma$-algebra on $[0,T]\times\Omega$. We will introduce some further notations. Let $L^{2,\nu}$ be the set of $\mathcal{B}\left(\mathbb{R}_{+}\right)$-measurable functions $\varphi:\mathbb{R}_{+}\rightarrow\mathbb{R}$ such that $\int_{\mathbb{R}_{+}} \varphi(y)^{2}\nu(dy)<\infty$. Let $L_{\mathcal{P}}^{2,\nu}(0,T;\mathbb{R})$ be the set of $\mathcal{P}\otimes\mathcal{B}\left(\mathbb{R}_{+}\right)$-measurable functions $\varphi:[0,T]\times\Omega\times\mathbb{R}_{+}\rightarrow\mathbb{R}$ such that $\mathbb{E}\int_{0}^{T}\int_{\mathbb{R}_{+}}|\varphi|^{2}\lambda\nu(dy)dt<\infty$. Let $L_{\mathcal{P}}^{\infty,\nu}(0,T;\mathbb{R})$ be the set of functions $\varphi\in L_{\mathcal{P}}^{2,\nu}(0,T;\mathbb{R})$ which are essentially bounded w.r.t. $dt \otimes d\mathbb{P}\otimes\nu(dy)$.

We consider the same financial market as in the previous subsection, but we assume instead the following:\\
\textbf{Assumption (C).}
\begin{equation*}
r\in L_{\mathbb{F}^{W},\mathbb{P}}^{2,\mathrm{BMO}}\left(0,T;\mathbb{R}\right),\quad \mu\in L_{\mathbb{F}^{W},\mathbb{P}}^{\infty}\left(0,T;\mathbb{R}^{n}\right),\quad \sigma\in L_{\mathbb{F}^{W},\mathbb{P}}^{\infty}\left(0,T;\mathbb{R}^{n\times n}\right),
\end{equation*}
and there exists a constant $\delta>0$ such that $\sigma \sigma^{\prime}\geqslant\delta I_{n\times n}$ for $t\in[0,T]$.

We assume that the investor is allowed to purchase reinsurances to control to control its exposure to the insurances risk. We now consider a reinsurance strategy $q_{t}$, and then the surplus process $V_{t}$ can be represented as follows:
\begin{equation}
\begin{aligned}
dV_{t}=(bq_{t}+a)dt-q_{t}\int_{\mathbb{R}^{+}}y\tilde{\gamma}(dt,dy)
\end{aligned}
\end{equation}
where $b>0,a$ are constants and more details can be seen in \cite{xu2024}. Then the investor's wealth process is given by
\begin{equation}\label{SDE4}
\left\{\begin{array}{l}
dX^{\pi,q}_{t}=\left(r_{t}X^{\pi,q}_{t}+\pi_{t}^{\prime}\mu_{t}+bq_{t}+a\right)dt+\pi_{t}^{\prime}\sigma_{t}dW_{t}-q_{t}\int_{\mathbb{R}^{+}}y\tilde{\gamma}(dt,dy),\\
X_{0}=x.
\end{array}\right.
\end{equation}
The class of admissible investment-insurance strategies is defined as the set
\begin{align*}
\bar{\Pi}:=\bigg\{(\pi,q)\in L_{\mathbb{F},\mathbb{P}}^{2}\left(0,T;\mathbb{R}^{n+1}\right)\mid q\geqslant 0\bigg\}.
\end{align*}
Under assumption (C), for any $(\pi,q)\in\bar{\Pi}$, similar to Lemma \ref{state space theorem}, one can easily show that $X$ is square-integrable. Denote
\begin{align*}
\bar{\mathcal{A}}:=\bigcup_{\epsilon>0}\bigg\{&(\eta,\psi)\in L_{\mathbb{F},\mathbb{P}}^{2}\left(0,T;\mathbb{R}^{n}\right)\times L_{\mathcal{P}}^{\infty,\nu}(0,T;\mathbb{R})\mid\psi\geqslant-1+\epsilon,\mathbb{E}\left[\left(\Lambda_{T}^{\eta, \psi}\right)^{2}\right]<\infty,\\
&\mathbb{E}_{\tau}\left[\int_{\tau}^{T}\left(\left|\eta_{s}\right|^{2}+\int_{\mathbb{R}_{+}}\left|\psi_{s}(y)\right|^{2}v(dy)\right)ds\right]\leqslant c\\
&\text{ for any stopping time $\tau\leqslant T$ and a constant $c>0$}\bigg\}
\end{align*}
and a probability
\begin{align*}
\left.\frac{d\mathbb{P}^{\eta,\psi}}{d\mathbb{P}}\right|_{\mathcal{F}_{t}}=\Lambda_{t}^{\eta,\psi}.
\end{align*}
where the Dol\'eans-Dade stochastic exponential
$$
\Lambda^{\eta, \psi}:=\mathcal{E}\left(\int_{0}^{t} \eta_{s}^{\prime} \mathrm{d} W_{s}+\int_{0}^{t} \int_{\mathbb{R}_{+}} \psi_{s}(y) \tilde{\gamma}(\mathrm{d} s, \mathrm{~d} y)\right)
$$
is a strictly positive martingale on $[0,T]$.

Except for $W^{\eta}_{t}$ in equation \eqref{Girsanov}, we also define
\begin{equation*}
\tilde{\gamma}^{\psi}(dt,dy)=\tilde{\gamma}(dt,dy)-\lambda\psi_{t}(y)\nu(dy)dt,
\end{equation*}
which is a compensated Poisson random measure under $\mathbb{P}^{\eta,\psi}$. We consider the following MMV and MV investment-reinsurance problems:
\begin{equation}\label{MMV-3}
\sup_{(\pi,q)\in\bar{\Pi}}\inf_{(\eta,\psi)\in\bar{\mathcal{A}}}\mathbb{E}^{\mathbb{P}^{\eta,\psi}}\left[X_{T}^{\pi,q}+\frac{1}{2\theta}\left(\Lambda_{T}^{\eta,\psi}-1\right)\right],
\end{equation}
and
\begin{equation}\label{MV-3}
\underset{(\pi,q)\in\bar{\Pi}}{\operatorname{sup}}\left[\mathbb{E}[X^{\pi,q}_{T}]-\frac{\theta}{2}\mathrm{Var}(X^{\pi,q}_{T})\right],
\end{equation}
where $\theta$ is a given positive constant. Without loss of generality, we may assume $a=0$ in \eqref{SDE4}.

\begin{corollary}
Suppose assumption (C) holds. The pair $(\hat{\eta},\hat{\psi},\hat{\pi},\hat{q})$ defined by
$$
\begin{aligned}
&\hat{\eta}=-\sigma^{-1}\mu-h^{-1}L,\\
&\hat{\psi}=\frac{by}{\lambda\int_{\mathbb{R}^{+}}y^{2}\nu(dy)},\\
&\hat{\pi}=\left(h \sigma^{\prime}\right)^{-1}\left(\frac{\Lambda^{\hat{\eta},\hat{\psi}}}{\theta}\left[\left(\sigma^{-1} \mu+h^{-1} L\right) Y-Z\right]-X^{\hat{\pi},\hat{q}} L\right),\\
&\hat{q}=\frac{b\Lambda^{\hat{\eta},\hat{\psi}}Y}{h\theta\lambda\int_{\mathbb{R}^{+}}y^{2}\nu(dy)}
\end{aligned}
$$
is optimal for the MMV investment-reinsurance problem \eqref{MMV-3} and $(\hat{\pi},\hat{q})$ is also the optimal strategy of MV investment-reinsurance problem \eqref{MV-3}, where $(h, L)$ and $(Y, Z)$ are the unique solutions of following BSDEs, respectively,
\begin{equation}
\left\{\begin{array}{l}
dh_{t}=\left[-r_{t}h_{t}+(\sigma^{-1}_{t}\mu_{t})^{\prime}L_{t}+\frac{L_{t}^{\prime}L_{t}}{h_{t}}\right]dt+L_{t}^{\prime}dW_{t},\\
h_{T}=1,
\end{array}\right.
\end{equation}
and
\begin{equation}\label{Y2}
\left\{\begin{array}{l}
dY_{t}=\left[-|\sigma_{t}^{-1}\mu_{t}+h_{t}^{-1}L_{t}|^{2}Y_{t}-\frac{b^{2}Y_{t}}{\lambda\int_{\mathbb{R}^{+}}y^{2}\nu(dy)}+2(\sigma_{t}^{-1}\mu_{t}+h^{-1}_{t}L_{t})^{\prime}Z_{t}\right]dt+Z_{t}^{\prime}dW_{t},\\
Y_{T}=1.
\end{array}\right.
\end{equation}
Moreover, the problem \eqref{MMV-3} and \eqref{MV-3} have the same optimal value
$$
h_{0}x+\frac{1}{2\theta}(Y_{0}-1).
$$
\end{corollary}
\begin{proof}
To begin with, the solvability of \eqref{Y2} is similar to \eqref{Y}. By It\^o's formula, it is easy to verify
\begin{equation}\label{optimal x-2}
\begin{aligned}
\theta &h_{t}X_{t}^{\hat{\pi},\hat{q}}+\Lambda_{t}^{\hat{\eta},\hat{\psi}}Y_{t}\equiv\theta h_{0}x+Y_{0},\\
i.e.\quad &X^{\hat{\pi},\hat{q}}_{t}=\frac{\theta h_{0}x+Y_{0}-\Lambda^{\hat{\eta},\hat{\psi}}_{t}Y_{t}}{\theta h_{t}},
\end{aligned}
\end{equation}
and we can use the same method in section \ref{section 4} to prove
$$
\mathbb{E}[(\Lambda_{t}^{\hat{\eta},\hat{\psi}})^{2}]<\infty,\quad t\in[0,T].
$$
We define
\begin{align*}
R_{t}^{(\eta,\psi,\pi,q)}:=h_{t}X_{t}^{\pi,q}+\frac{1}{2 \theta}(\Lambda_{t}^{\eta,\psi}Y_{t}-1),\quad t\in[0,T],~(\eta,\psi,\pi,q)\in\bar{\mathcal{A}}\times\bar{\Pi},
\end{align*}
and similar to \eqref{R}, we have
\begin{align*}
dR_{t}^{(\eta,\psi,\pi,q)}=&\left[\frac{\Lambda}{2\theta}\left(Y|\eta|^{2}+2\eta^{\prime}Z+\frac{2\theta}{\Lambda}h\eta^{\prime} \sigma^{\prime} \pi\right)+\frac{\Lambda}{2 \theta}\int_{\mathbb{R}_{+}}\left(Y\psi^{2}-\frac{2\theta}{\Lambda}hqy\psi\right)\lambda\nu(dy)\right.\\
&\quad+\left(L^{\prime}\sigma^{-1}\mu+h^{-1}L^{\prime}L+L^{\prime}\eta\right)X+\pi^{\prime}(h\sigma\mu+\sigma L)+hbq\\
&\left.\quad+\frac{\Lambda}{2\theta}\left(-|\sigma^{-1}\mu+h^{-1}L|^{2}Y-\frac{b^{2}Y}{\lambda\int_{\mathbb{R}^{+}}y^{2}\nu(dy)}+2(\sigma^{-1}\mu+h^{-1}L)^{\prime}Z\right)\right]dt\\
&\quad+(\cdots)dW^{\eta}+\int_{\mathbb{R}_{+}}(\cdots)\tilde{\gamma}^{\psi}(dt,dy).
\end{align*}

Firstly, for an increasing sequence of stopping times $\tau_{n}\uparrow+\infty$, we have
\begin{align*}
\mathbb{E}^{\hat{\eta},\hat{\psi}}\left[R_{T\wedge{\tau_{n}}}^{(\hat{\eta},\hat{\psi},\pi,q)}\right]=&h_{0}x+\frac{1}{2\theta}(Y_{0}-1)+\int_{0}^{T\wedge{\tau_{n}}}\bigg[\frac{\Lambda}{2\theta}\bigg(|\hat{\eta}|^{2}Y-|\sigma^{-1}\mu+h^{-1}L|^{2}Y\bigg)\\
&+\frac{\Lambda}{2\theta}\bigg(2\hat{\eta}^{\prime}Z+2(\sigma^{-1}\mu+h^{-1}L)^{\prime}Z\bigg)+\bigg(h\hat{\eta}^{\prime}\sigma^{\prime} \pi+\pi^{\prime}(h\sigma\mu+\sigma L)\bigg)\\
&+\left(L^{\prime}\sigma^{-1}\mu+h^{-1}L^{\prime}L+L^{\prime}\hat{\eta}\right)X+\frac{\Lambda}{2 \theta}\int_{\mathbb{R}_{+}}\left(\hat{\psi}^{2}Y-\frac{b^{2}Y}{\lambda^{2}\int_{\mathbb{R}^{+}}y^{2}\nu(dy)}\right)\lambda\nu(dy)\\
&+hq\left(b-\int_{\mathbb{R}_{+}}y\hat{\psi}\lambda\nu(dy)\right)\bigg]dt\\
=&h_{0}x+\frac{1}{2\theta}(Y_{0}-1).
\end{align*}
Similar to \eqref{local}, we can prove $\mathbb{E}\left[\sup _{t\in[0,T]}\left|\Lambda_t^{\hat{\eta},\hat{\psi}} R_t^{\hat{\eta},\hat{\psi}, \pi,q}\right|\right]<\infty$. Then we let $n\rightarrow\infty$ and obtain
$$
\mathbb{E}^{\mathbb{P}^{\hat{\eta},\hat{\psi}}}\left[X_{T}^{\pi,q}+\frac{1}{2\theta}\left(\Lambda_{T}^{\hat{\eta},\hat{\psi}}-1\right)\right]=h_{0}x+\frac{1}{2\theta}(Y_{0}-1).
$$

Secondly, similar to \eqref{local-2}, we have
\begin{align*}
\mathbb{E}^{\eta,\psi}\left[R_{T\wedge{\tau_{n}}}^{(\eta,\psi,\hat{\pi},\hat{q})}\right]=&h_{0}x+\frac{1}{2\theta}(Y_{0}-1)+\mathbb{E}^{\mathbb{P}^{\eta,\psi}}\int_{0}^{T\wedge\tau_{n}}\frac{\Lambda}{2\theta}\bigg[Y\left|\eta+\frac{1}{Y}\left(Z+\frac{\theta (h\sigma^{\prime}\hat{\pi}+\hat{X}L)}{\Lambda}\right)\right|^{2}\\
&\qquad+Y\int_{\mathbb{R}^{+}}\left|\psi-\frac{\theta h\hat{q}y}{\Lambda Y}\right|^{2}\lambda\nu(dy)\bigg]dt\\
\geqslant&h_{0}x+\frac{1}{2\theta}(Y_{0}-1).
\end{align*}
We can also prove $\mathbb{E}\left[\sup _{t\in[0,T]}\left|\Lambda_t^{\eta,\psi} R_t^{\eta,\psi, \hat{\pi},\hat{q}}\right|\right]<\infty$ and then $\mathbb{E}^{\eta,\psi}\left[R_{T}^{(\eta,\psi,\hat{\pi},\hat{q})}\right]\geqslant h_{0}x+\frac{1}{2\theta}(Y_{0}-1)$.

Thirdly, it is obvious that
$$
\mathbb{E}^{\mathbb{P}^{\hat{\eta},\hat{\psi}}}\left[X_{T}^{\hat{\pi},\hat{q}}+\frac{1}{2\theta}\left(\Lambda_{T}^{\hat{\eta},\hat{\psi}}-1\right)\right]=h_{0}x+\frac{1}{2\theta}(Y_{0}-1).
$$

At last, we recall the notation in \eqref{F(K)}-\eqref{F(K)2} and
similar to \eqref{XT-gamma2}, we have
\begin{equation*}
\begin{aligned}
&\mathbb{E}[(X_{T}^{\pi,q}-\gamma)^{2}]\\
&=Y_{0}^{-1}(h_{0}x-\gamma)^{2}+\mathbb{E}\bigg[\int_{0}^{T}\bigg\{Y_{t}^{-1}\left|h_{t}\sigma_{t}^{\prime}\pi_{t}+X_{t}L_{t}+\left(\phi_{t}-\frac{Z_{t}}{Y_{t}}\right)(h_{t}X_{t}^{\pi,q}-\gamma)\right|^{2}\\
&\qquad+Y_{t}^{-1}\int_{\mathbb{R}^{+}}y^{2}\lambda\nu(dy)\left|h_{t}q_{t}+\frac{(h_{t}X_{t}-\gamma)b}{\int_{\mathbb{R}^{+}}y^{2}\lambda\nu(dy)}\right|^{2}\bigg].
\end{aligned}
\end{equation*}
Therefore the optimal feedback strategy is
\begin{align*}
&h_{t}\sigma_{t}^{\prime} \hat{\pi}_{t}^{\gamma(K)}+X_{t}^{\hat{\pi}^{\gamma(K)},\hat{q}^{\gamma(K)}}L_{t}=-\left(\phi_{t}-\frac{Z_{t}}{Y_{t}}\right)\left(h_{t}X_{t}^{\hat{\pi}^{\gamma(K)},\hat{q}^{\gamma(K)}}-\gamma(K)\right),\\
&h_{t}\hat{q}_{t}=-\frac{\left(h_{t}X_{t}^{\hat{\pi}^{\gamma(K)},\hat{q}^{\gamma(K)}}-\gamma(K)\right)b}{\int_{\mathbb{R}^{+}}y^{2}\lambda\nu(dy)},
\end{align*}
and furthermore
\begin{align*}
\hat{\gamma}(K) = \frac{Y_{0}^{-1}h_{0}x-K}{Y_{0}^{-1}-1},\quad\hat{K}= h_{0}x+\frac{1-Y_{0}^{-1}}{\theta Y_{0}^{-1}},
\end{align*}
and
\begin{equation*}
\underset{(\pi,q)\in\bar{\Pi}}{\operatorname{sup}}\left[\mathbb{E}[X_{T}]-\frac{\theta}{2}\mathrm{Var}(X_{T})\right]=h_{0}x+\frac{1}{2\theta}(Y_{0}-1).
\end{equation*}
From \eqref{optimal x-2}, we can verify that
\begin{align*}
\hat{\pi}_{t}
& = (h_{t}\sigma^{\prime}_{t})^{-1}\left[-(\phi_{t}-\frac{Z_{t}}{Y_{t}})\left(h_{t}X_{t}^{\hat{\pi},\hat{q}}-h_{0}x-\frac{Y_{0}}{\theta}\right)-X_{t}^{\hat{\pi},\hat{q}}L_{t}\right]\\
& = (h_{t}\sigma_{t}^{\prime})^{-1}\left[(\phi_{t}Y_{t}-Z_{t})\left(\frac{\theta h_{0}x+Y_{0}-\theta h_{t}X_{t}^{\hat{u}}}{\theta Y_{t}}\right)-X_{t}^{\hat{\pi},\hat{q}}L_{t}\right]\\
& = (h_{t}\sigma_{t}^{\prime})^{-1}\left[\frac{\Lambda^{\hat{\eta},\hat{\psi}}}{\theta}\big(\phi_{t}Y_{t}-Z_{t}\big)-X_{t}^{\hat{\pi},\hat{q}}L_{t}\right]
\end{align*}
and
\begin{align*}
\hat{q}_{t}
&=-\frac{\left(h_{t}X_{t}^{\hat{\pi},\hat{q}}-h_{0}x-\frac{Y_{0}}{\theta}\right)b}{h_{t}\int_{\mathbb{R}^{+}}y^{2}\lambda\nu(dy)}=\frac{\left(\theta h_{0}x+Y_{0}-\theta h_{t}X_{t}^{\hat{\pi},\hat{q}}\right)b}{\theta h_{t}\int_{\mathbb{R}^{+}}y^{2}\lambda\nu(dy)}=\frac{\Lambda^{\hat{\eta},\hat{\psi}}Y_{t}b}{\theta h_{t}\int_{\mathbb{R}^{+}}y^{2}\lambda\nu(dy)}.
\end{align*}
\end{proof}
\begin{remark}
    Under the Cram\'{e}r-Lundberg model, constrained mean-variance and monotone mean-variance investment-reinsurance problems with random coefficients have been studied in \cite{xu2024a} and \cite{xu2024} respectively. Without trading constraints and under the restriction that $r,\mu,\sigma$ are adapted with respect to $\mathbb{F}^{W}$, our results complement their results by further allowing the interest rate to be random and unbounded. The general case is left for future research.
\end{remark}

\bibliographystyle{siam}
\bibliography{bib}

\end{document}